\documentclass[a4paper, reqno]{amsart}
\usepackage{enumerate}
\usepackage{pb-diagram}
\usepackage{microtype}
\usepackage{amsmath}
\usepackage{amssymb, pb-diagram}
\usepackage[all]{xy}
\usepackage{graphicx} 
\usepackage{mathabx}
\usepackage{chapterbib}
\usepackage{epsfig}
\usepackage{amsfonts}
\usepackage{amssymb}
\usepackage{amsmath}   
\usepackage{amsthm}
\usepackage{color}
\usepackage{latexsym}
\usepackage{hyperref}

\newtheorem{theorem}{Theorem}[section]
\newtheorem{proposition}[theorem]{Proposition}
\newtheorem{lemma}[theorem]{Lemma}
\newtheorem{corollary}[theorem]{Corollary}
\newtheorem{observation}[theorem]{Observation}

\newtheorem{question}[theorem]{Question}

\newtheorem*{Theorem A}{Theorem A}
\newtheorem*{Theorem B}{Theorem B}

\theoremstyle{definition}
\newtheorem{DEF}[theorem]{Definition}
\newtheorem{remark}[theorem]{Remark}
\newtheorem{example}[theorem]{Example}

\makeatletter

\newcommand{\ul}[1]{\underline{#1}}

\newcommand{\N}{\mathbb{N}}
\newcommand{\Z}{\mathbb{Z}}
\newcommand{\Q}{\mathbb{Q}}

\newcommand{\frakA}{\mathfrak{A}}
\newcommand{\frakF}{\mathfrak{F}}

\newcommand{\calA}{\mathcal{A}}
\newcommand{\calB}{\mathcal{B}}

\newcommand{\calF}{\mathcal{F}}

\newcommand{\calM}{\mathcal{M}}
\newcommand{\calO}{\mathcal{O}}

\renewcommand{\:}{\colon}

\DeclareMathOperator{\Ab}{\mathfrak{Ab}}

\DeclareMathOperator{\Mod}{Mod}

\newcommand{\ModR}{\mathop{{\operator@font Mod}\text{-}R}}
\newcommand{\ModG}{\mathop{{\operator@font Mod}\text{-}G}}
\newcommand{\RMod}{\mathop{R\text{-}{\operator@font Mod}}}
\newcommand{\GMod}{\mathop{G\text{-}{\operator@font Mod}}}

\DeclareMathOperator{\Tor}{Tor}
\DeclareMathOperator{\Ext}{Ext}

\DeclareMathOperator{\mor}{mor}
\DeclareMathOperator{\Hom}{Hom}

\renewcommand{\Im}{\operatorname{Im}}
\DeclareMathOperator{\coker}{coker}

\newcommand{\OC}[2]{\mathop{\calO_{#2}#1}\nolimits}

\newcommand{\OFG}{\OC{G}{\frakF}}

\newcommand{\OFGMod}{\mathop{\calO_{\frakF}G\text{-}{\Mod}}}

\newcommand{\CRMod}{{{\OFG}\text{-}{\Mod_R}}}

\newcommand{\RModC}{\mathop{{\Mod_R\text{-}}{\OFG}}}

\newcommand{\Fall}{\frakF_{\operator@font all}}
\DeclareMathOperator{\fin}{fin}
\newcommand{\Ffin}{\frakF_{\fin}}
\DeclareMathOperator{\vc}{vc}
\newcommand{\Fvc}{\frakF_{\vc}}
\newcommand{\Fic}{\frakF_{\operator@font ic}}
\newcommand{\Ffg}{\frakF_{\operator@font fg}}
\newcommand{\Fpc}{\frakF_{\operator@font pc}}
\newcommand{\Fab}{\frakF_{\operator@font ab}}
\newcommand{\Fvpc}{\frakF_{\operator@font vpc}}
\newcommand{\Fvab}{\frakF_{\operator@font vab}}

\DeclareMathOperator{\FP}{FP}

\DeclareMathOperator{\cd}{cd}
\DeclareMathOperator{\hd}{hd}

\DeclareMathOperator{\pd}{pd}
\DeclareMathOperator{\fld}{fld}

\DeclareMathOperator{\ucd}{\ul{cd}}
\DeclareMathOperator{\uhd}{\ul{hd}}

\newcommand{\onto}{\twoheadrightarrow}

\newcommand{\isom}{\cong}

\DeclareMathOperator{\id}{id}

\newcommand{\member}{\mathbin{\in_{\mathrm{m}}}}

\makeatother

\newcounter{mynotecounter}

\newenvironment{mynote}{\color{red}\begin{sf}\noindent%
\refstepcounter{mynotecounter}%
\textbf{Note~\arabic{mynotecounter}:}}{\end{sf}}

\begin{document}

\title[Bredon (co)homological dimension]{A note on the Mittag--Leffler condition for Bredon-modules }


\author{Martin G. Fluch}
\address{MATHEMA Software GmbH, Henkestra{\ss}e 91, 91052 Erlangen, Germany}
\email{martin.fluch@gmail.com}

\author{Giovanni Gandini}
\address{K{\o}benhavns Universitet, Institut for Matematiske Fag, Universitetsparken 5, 2100 K{\o}benhavn {\O}, Denmark}
\email{ggandini@math.ku.dk}

\author{Brita  Nucinkis}
\address{ Department of Mathematics,
Royal Holloway, University of London, 
Egham, TW20 0EX }
\email{Brita.Nucinkis@rhul.ac.uk}


\date{\today}



\begin{abstract}
In this note we show the Bredon-analogue of a result by Emmanouil and Talelli, which gives a criterion when the homological and cohomological dimensions of a countable group~$G$ agree. We also present some applications to groups of Bredon-homological dimension~$1$.
\end{abstract}

\maketitle

\section{introduction}

\noindent It is a well-known fact that for a countable group, the homological dimension and the  cohomological dimension over a non-zero commutative ring $R$ differ by at most one, and that in general the homological dimension is always less than or equal to the cohomological dimension, see, for example \cite[Theorem 4.6]{bieribook}. Emmanouil and Talelli \cite[Theorem 2.1]{emmanouil-12} give a criterion  for a group of finite integral homological dimension to have both quantities  equal. The main ingredient for this result is the following:

\begin{theorem}\label{main-et}\cite[Theorem 1.3]{emmanouil-12}
Let $R$ be a countable ring and $M$ be a countably generated flat left $R$-module. Then the following conditions are equivalent:
\begin{enumerate}
\item $M$ is projective.
\item $\Ext_R^1(M,R)=0$.
\item $\Ext_R^1(M,R)$ is a countable group.
\end{enumerate}
\end{theorem}

In this paper we will verify a  generalisation of this theorem to the setting of Bredon-cohomology. Here the group $G$ is replaced with the orbit category ${\OFG}$, whose
objects are the transitive $G$-sets $G/H$ with $H\in \frakF$ and whose
morphisms are $G$-maps. The category $\OFGMod$ of right Bredon modules is the
functor category whose objects are contravariant functors $M\: {\OFG}\to \Ab$
from the orbit category to the category $\Ab$ of abelian groups. As a functor
category  it is, for example,  an abelian category which satisfies the same
Grothendieck Axioms as the category $\Ab,$ and which has enough projective and 
enough injective objects. Hence there is a natural setting to do homological algebra.
In the next section we briefly introduce those notions from Bredon-cohomology needed here. We refer the reader to~\cite{luck-89, mislin-03} and also 
to the PhD thesis of the first author~\cite{fluch-phdthesis} for more detail. We prove:

\begin{Theorem A}
Let $G$ be a countable group, $\frakF$ a family of subgroups of $G$ of finite 
conjugacy type and let~$\calF$ be a complete set of representatives of the 
conjugacy classes of the elements of~$\frakF$. If $n := \hd_\frakF G$ is 
finite, then the following statements are equivalent:

\begin{enumerate}
\item $\cd_\frakF G = n $

\medskip

\item $H^{n+1}_\frakF \Bigl(G; \prod\limits_{H\in \calF} \Z[-,
G/H]_G\Bigr) = 0$

\item $H^{n+1}_\frakF \Bigl(G; \prod\limits_{H\in \calF} \Z[-,
G/H]_G\Bigr)$ is countable.
\end{enumerate}
\end{Theorem A}

Note that for $\frakF=\{1\}$ this is
\cite[Theorem 2.1]{emmanouil-12}.

The main ingredients needed in the proof are the facts that a flat contravariant $\OFG$-module is a colimit of finitely generated free $\OFG$-modules \cite[Theorem~3.2]{nucinkis-04}, and that certain limits of duals of free $\OFG$-modules satisfy the Mittag--Leffler condition for Bredon-modules. This is defined analogously to the case of modules over a commutative ring, and we shall give precise definitions and properties in Section \ref{ML-section}.

We will apply Theorem A to groups of Bredon-homological dimension $1$. It is a well-known conjecture, attributed to R.~Bieri, that a finitely generated group of integral homological dimension $1$ is free. In \cite{bridson-12} Bridson and Kropholler remark that  a naive generalisation to rational homological dimension is not true as the case of the lamplighter group shows. In Section \ref{Bredonhom-1} we ask a related  question for Bredon homology, where we replace the condition that the group is finitely generated by the property Bredon-$\FP_1$. We shall answer the question in some specific situations.

Finally, we  apply Theorem \ref{main-et} to the rational cohomology of the group and will show a virtual version of \cite[Proposition 3.2]{emmanouil-12}.

\subsection*{Acknowledgements}

The authors would like to thank the referee for many helpful comments and for spotting an error in an earlier version. The first author gratefully acknowledges the support through the SFB~701 in Bielefeld. The second author acknowledges the support by the Danish National Research Foundation (DNFS) through the Centre for Symmetry and Deformation.

\section{Preliminaries on Bredon cohomology}

Let $G$ be a group and let $\frakF$ be a family of subgroups closed under taking subgroups and under conjugation. ${\OFG}$ denotes the  category with elements the transitive $G$-sets with stabilisers in $\frakF$ and morphisms, $\mor_{\frakF}(x,y,)$ given by the set of $G$-maps between them.
We now form the free abelian group on the set of morphisms, which is denoted as follows:
$$\Z\mor_{\frakF}(x,y) =\Z[x,y]_{\frakF}.$$

Let $R$ denote a commutative ring with $1$.
The category of covariant ${\OFG}$-modules, denoted $\RModC$, is now defined to be the category of covariant additive functors from ${\OFG}$ to the category of left $R$-modules. Analogously we define the category of contravariant ${\OFG}$-modules, denoted $\CRMod$. If the variance is clear from the context, or a statement is valid for either category, we simply talk of ${\OFG}$-modules.

The category of ${\OFG}$-modules inherits all of Grothendieck's axioms for an abelian category that are satisfied by the category of $R$-modules. Short exact sequences are evaluated point-wise, so are limits and colimits. We have the usual categorical tensor-product: let $M\in \RModC$ and $N \in \CRMod$, then the tensor product is denoted by 
$$M \otimes_{\frakF} N.$$

We can form contravariant and covariant ${\OFG}$-modules $R[-,y]_{\frakF}$ and $R[x,-]_{\frakF}$ respectively:
$$R[-,y]_{\frakF}(x) = R\otimes_\Z \Z[x,y]_{\frakF}\phantom{.}$$
and 
$$R[x,-]_{\frakF}(y) = R\otimes_\Z \Z[x,y]_{\frakF}.$$

\subsection{Free ${\OFG}$-modules} We have the usual notion of free, projective and flat modules. In particular, the modules $R[-,y]_{\frakF}$ and $R[x,-]_{\frakF}$ are free, and any free is a direct sum of modules of this form. For detail the reader is referred to \cite{luck-89} for the general set-up, and to \cite{fluch-phdthesis} and \cite{nucinkis-04} for this specific case.

Cohomology and Homology functors $\Ext_{\frakF}^*(M,N)$ and $\Tor_*^{\frakF}(M,N)$ are now defined in the usual manner.  
We say a free $\OFG$-module $F$ is finitely generated, or countably generated, if there is a $G$-finite, respectively $G$-countable $G$-set $\Delta$ with finite stabilisers, such that~$F \cong \Z[-,\Delta]_\frakF$.  For detail see \cite{luck-89, kropholler-09}.  We say an $\OFG$-module is finitely generated, countably generated if there is a finitely generated, respectively countably generated free module mapping onto it.

\subsection{Categories of finite type}

We say the category ${\OFG}$  is of \emph{finite type}, if there are finitely many objects $X=\{x_1,\ldots ,x_n\}$ such that for every object $y \in {\OFG}$ there exists $x_i \in X$ such that $\mor_{\OFG}(y,x_i) \neq 0$.
We say ${\OFG}$ is of \emph{finite isomorphism type} if there are finitely many objects $X=\{x_1,\ldots ,x_n\}$ and for each $y \in {\OFG}$ there is an isomorphism $\varphi \in \mor_{\OFG}(y,x_i)$ for some $x_i \in X$. This is equivalent to saying that $\OFG$ is of finite \emph{conjugacy type}, i.e.~that there are finitely many conjugacy classes of subgroups in $\frakF$.

\begin{remark} 
For $\frakF = \Ffin$ the family of finite subgroups, finite type and finite isomorphism type are the same and are equivalent to saying that the group is of type Bredon-$\FP_0$, i.e.~that the group has finitely many conjugacy classes of finite subgroups.

For $\frakF=\Fvc$ the family of virtually cyclic subgroups, on the other hand, Bredon-$\FP_0$ is equivalent to being of finite type \cite[Lemma 2.3]{desiconchabrita-11}, and it was shown in \cite{groveswilson} that an elementary amenable group of type Bredon-$\FP_0$ is necessarily virtually cyclic. Finite isomorphism type, on the other hand, is a much stronger condition. For example, the famous  of  construction Higmann--Neumann--Neumann \cite[6.4.6]{robinson}, is a group of finite type; in fact $\OFG$ has an initial object, but one can construct examples where the group is not of finite isomorphism type, see also \cite[Remark 2.1]{desiconchabrita-11}.

\end{remark}

\subsection{Dualising ${\OFG}$-modules}
To generalise the usual notion of a dual module over a ring  to this
setting one needs to observe that ${\OFG}$-modules can be described as modules 
over rings over several objects in the sense of~\cite{street-95}. 

We replace the ring $R$ by a bi-functor
$$R{\OFG}\: {\OFG}^{\rm opp} \otimes {\OFG} \to \RMod,$$
given by $R{\OFG}(x,y)=R[x,y]_{\frakF}$. In \cite{street-95} this functor is denoted $\mathcal{H}_{\OFG}$.

\begin{DEF}
The \emph{dual module $M^*$} of a contravariant ${\OFG}$-module $M$ is defined as
\begin{equation*}
M^* := \Hom_{\OFG}(M, R{\OFG}),
\end{equation*}
which is a covariant ${\OFG}$-module. If $M$ is a covariant ${\OFG}$-module, 
then its \emph{dual module~$M^*$} is defined in the same way, which yields a 
contravariant ${\OFG}$-module.
\end{DEF}

In particular this  defines additive contravariant functors
\begin{align*}
(-)^*\: & \CRMod \to \RModC
\\
\intertext{and}
(-)^*\: & \RModC \to \CRMod
\end{align*}
which are exact. In particular injective morphisms are mapped to surjective 
ones and vice versa.

\begin{observation}
Using Yoneda type arguments analogously to those in \cite{mislin-03, nucinkis-04} it follows directly that for 
any $x\in {\OFG}$ the module dual to $M= R[x, -]_{\frakF}$ is $M^* = R[-,x]_{\frakF}$ and that the dual module of~$M = R[-, x]_{\frakF}$ is $M^*= R[x,-]_{\frakF}$. 
\end{observation}

\begin{proposition}
\label{prop:dualising-bredon-modules}
Let $P$ be a finitely generated projective covariant ${\OFG}$-module.
\begin{enumerate}
\item $P^*$ is a finitely generated projective contravariant ${\OFG}$-module.

\item For any covariant ${\OFG}$-module $M$, there is a natural isomorphism
\begin{equation*}
\varphi\: P^*\otimes_{\OFG} M \to \Hom_{\OFG}(P,M)
\end{equation*}
of abelian groups.

\item For any contravariant ${\OFG}$-module $M$, there is a natural isomorphism
\begin{equation*}
\varphi'\: M\otimes_{\OFG} P \to \Hom_{\OFG}(P^*,M) \,.
\end{equation*}

\item There is a natural isomorphism
\begin{equation*}
 \varphi''\: P\to P^{**}. 
\end{equation*}
\end{enumerate}
\end{proposition}

\begin{proof} The proof of the analogous statement for ordinary cohomology ~\cite[Proposition I.8.3.]{brown-82} works nearly verbatim  
in the Bredon setting. This was checked for finitely generated free Bredon-modules in \cite{nucinkis-04}.  The result for projectives follows from the fact, that for finitely generated projective modules $P$, there is a finitely generated free $F$ and a module $Q$, such that $F \cong P\oplus Q$, and $F^* \cong P^* \oplus Q^*.$
\end{proof}

\subsection{Flat ${\OFG}$-modules} The following result was shown in  \cite[Theorem 3.2.]{nucinkis-04} for $\frakF =\Ffin,$ but the proof works for arbitrary families $\frakF.$

\begin{proposition}\label{lazard}
Let $M$ be a contravariant $\OFG$-module. Then the following are equivalent:
\begin{enumerate}
\item $M$ is flat.
\item $M$ is a colimit of finitely generated free $\OFG$-modules.
\end{enumerate}
\end{proposition}

\section{The Mittag--Leffler Condition}\label{ML-section}

The directed sets we consider in this article are all countable. Hence, when considering limits 
or colimits of objects~$\{A_\alpha\}_{\alpha\in D}$ directed by a countable 
directed set $D$, we may assume that  
we are dealing with a \emph{tower}
\begin{equation*}
\ldots \to A_4 \to A_3 \to A_2 \to A_1 \to A_0
\end{equation*}
or respectively with a \emph{co-tower} 
\begin{equation*}
\ldots \leftarrow A_4 \leftarrow A_3 \leftarrow A_2 \leftarrow A_1 \leftarrow 
A_0
\end{equation*}
of objects.

\begin{DEF}
\cite[Definition 3.5.6.]{weibel-94}
\label{def:m-l}
Let $\{A_\alpha\}$ be a tower of objects in an abelian category.

\begin{enumerate}
\item The tower of objects is said to satisfy the \emph{Mittag--Leffler 
condition,} if for any $k$ there exists a $j\geq k$ such that for any $i\geq j$ 
the image of $A_i\to A_k$ equals the image of $A_j\to A_k$.

\item The tower of objects is said to satisfy the \emph{strict (or trivial) Mittag--Leffler 
condition,} if for any $k$ there exists a $j\geq k$ such that the map $A_j\to 
A_k$ is zero.
\end{enumerate}
\end{DEF}

If the abelian category $\mathfrak{A}$  is cocomplete and has enough 
injectives, then one can define the right derived functors of the 
$\varprojlim$-functor. If in addition, the category $\frakA$ satisfies the Grothendieck Axiom (AB4*) then
the following explicit construction of the right derived functors of
$\varprojlim$ works for \emph{countable} limits, that is limits of towers, see \cite[\S 3.5.]{weibel-94}:

\begin{enumerate}
\item $\varprojlim^0 A_\alpha := \varprojlim A_\alpha$

\item $\varprojlim^1 A_\alpha := \coker(\Delta)$ where $\Delta\:
\prod A_\alpha \to \prod A_\alpha$ is the homomorphism as defined
in~\cite[p.~81]{weibel-94}

\item $\varprojlim^n
A_\alpha := 0$ for $n\geq 2$.
\end{enumerate}

\begin{DEF}
\label{def:m-l-weak}
Consider the functor category $\frakA^I$ where $\frakA$ is an abelian category
and $I$ is a small category. Let $\{A_\alpha\}$ be a tower of objects in
$\frakA^I$.

We say that $\{A_\alpha\}$ satisfies the (strict) Mittag--Leffler condition in 
the \emph{weak sense,} if for every $i\in I$ the tower $\{A_\alpha(i)\}$ of 
objects in $\frakA$ statisfies the (strict) Mittag--Leffler condition.

If we want to emphasise that the tower $\{A_\alpha\}$ satisfies the (strict) 
Mittag--Leffler condition not just in the weak sense, but as defined in 
Definition~\ref{def:m-l}, then we may say that the (strict) Mittag--Leffler 
condition is satisfied in the \emph{strong sense}.
\end{DEF}

\begin{observation}
\label{obs:key-obs}
Let $\{A_\alpha\}$ be a tower of objects in $\frakA^I$ which satisfies the 
(strict) Mittag--Leffler condition in the weak sense. If $I$ has only finitely 
many isomorphism classes of objects, then $\{A_\alpha\}$ satisfies the (strict) 
Mittag--Leffler condition in the strong sense. 
\end{observation}

\section{Proof of Theorem A}

The following, preliminary result  is essentially  \cite[p.~16, statement~3.1.3]{raynaud-71} translated to our setting. Since it needs some additional argument,   we include a detailed proof.

\begin{proposition}
\label{prop:raynaud-71}
Let $P$ be a contravariant ${\OFG}$-module. Let 
\begin{equation*}
\ldots \leftarrow L_{4} \leftarrow L_{3} \leftarrow L_{2} \leftarrow
L_{1}\leftarrow L_{0}
\end{equation*}
be a co-tower of finitely generated free contravariant
${\OFG}$-modules such that $P=\varinjlim L_\alpha$.
Then the following statements are equivalent:

\begin{enumerate}
\item\label{prop:raynaud-71-1}
$P$ is projective.

\item\label{prop:raynaud-71-2}
The tower
\begin{equation*}
\ldots \to L_3^* \to L_2^* \to L_1^* \to L_0^*
\end{equation*}
satisfies the Mittag--Leffler condition in the strong sense.
\end{enumerate}
\end{proposition}

\begin{proof} 
\ref{prop:raynaud-71-1} $\implies$ \ref{prop:raynaud-71-2}:  Since every projective is a direct summand of a free and by Proposition~\ref{lazard}, we can assume that $P$ is free on the ${\OFG}$-set $X$.
Now fix a $k\geq 0$ and assume that $L_k$ is finitely generated by the ${\OFG}$-set $Y_k$. Let $X_k \subseteq X$ be a $\OFG$-subset big enough so that the image of $L_k$ in $P$ is contained in $P_k$ the free $\OFG$-module generated by $X_k$. Note that $P_k$ is a direct summand of $P$ and that
 the canonical morphism $L_k \to P$ factors through  $P_k.$ 
 For each~$x \in X_k$ there is an~$i_x \geq k$ such that $x \in P$ is in the image of the canonical morphism $L_{i_x} \to P$. Now put  $i = \max\{i_x \, |\, x \in X_k\}$. Hence, for each $j\geq i$  the canonical map $L_k \to L_j$ factors through $P_k$.

Dualising now yields a surjection $P^* \onto P^*_k$ and the map $L_j^* \to L_k^*$ factors through~$P_k^*$ such that $L_j^* \onto P_k^*$.
Hence 
$$\Im(L_j^* \to L_k^*) = \Im(P_k^* \to L_k^*)$$
is independent of $j$ and hence the claim follows.

\smallskip

\ref{prop:raynaud-71-2} $\implies$ \ref{prop:raynaud-71-1}:  Since $- \mathbin{\otimes_{\OFG}} M$ is a right exact 
functor it follows  that the 
tower $(L_\alpha^* \mathbin{\otimes_{\OFG}} M)$ of abelian groups satisfies 
the Mittag--Leffler condition for any left ${\OFG}$-module~$M$. The natural 
isomorphism $L_\alpha^*\mathbin{\otimes_{\OFG}} M \isom \Hom_{\OFG}(L_\alpha, 
M)$ of Proposition~\ref{prop:dualising-bredon-modules} implies, that 
$(\Hom_{\OFG}(L_\alpha, M))$ is a tower satisfying  the 
Mittag--Leffler condition as well.

Consider the sequence of natural isomorphisms
\begin{equation*}
\Hom_{\OFG}(P, M) 
\isom
\Hom_{\OFG}(\varinjlim L_\alpha, M)
\isom
\varprojlim \Hom_{\OFG}(L_\alpha, M)
\end{equation*}
which in turn yields an isomorphism
\begin{equation*}
\Ext_{\OFG}^1(P,M) \isom \varprojlim\nolimits^1 \Hom_{\OFG}(L_\alpha, M).
\end{equation*}
Since $(\Hom_{\OFG}(L_\alpha, M))$ is a tower of abelian groups which satisfies 
the Mittag--Leffler condition, it follows that $\varprojlim\nolimits^1 
\Hom_{\OFG}(L_\alpha, M)=0$, see for example \cite[Proposition~3.5.7]{weibel-94}. Therefore $\Ext_{\OFG}^1(P,M)=0$ and since this is 
true for arbitrary~$M$ it follows that $P$ is projective.
\end{proof}

\begin{lemma}\label{lem:cat-gray-66} Let $(A_\alpha)$ be a tower of ${\OFG}$-modules, such that $A_\alpha(x)$ is countable for each $\alpha \in \N$ and $x \in {\OFG}$. Then the following 
conditions are equivalent:
\begin{enumerate}
\item The tower $(A_\alpha)$ satisfies the Mittag--Leffler condition in the weak sense.

\item $\varprojlim^1 A_\alpha(x) = 0$ for every $x \in \OFG.$

\item $\varprojlim^1 A_\alpha(x)$ is a countable $R$-module for every $x \in \OFG.$

\end{enumerate}
\end{lemma}

\begin{proof} This was proved in  \cite{gray-66} for abelian groups.  The result now follows by applying that result to every $x \in \OFG.$
\end{proof}

\begin{proposition}\label{propA}
Let ${\OFG}$ be of finite isomorphism type and 
let $M$ be a countably generated flat contravariant ${\OFG}$-module. Then, the following 
conditions are equivalent:
\begin{enumerate}
\item\label{propA-1} $M$ is projective.

\item\label{propA-2} $\Ext_{\OFG}^1\Bigl(M; \prod\limits_{x\in {\OFG}_0} \Z[-,x]\Bigr)=0$.

\item\label{propA-3} $\Ext_{\OFG}^1\Bigl(M; \prod\limits_{x\in {\OFG}_0} \Z[-,x]\Bigr)$ is a countable abelian group.
\end{enumerate}
\end{proposition}

\begin{proof}
As before we have that \ref{propA-1} $\implies$ \ref{propA-2} $\implies$ \ref{propA-3}. To show \ref{propA-3} $\implies$ \ref{propA-1} 
we follow the same outline as the proof of \cite[Theorem 1.3]{emmanouil-12} with some adaptation to the setting of ${\OFG}$-modules.
Since $M$ is countably generated and flat, it is the direct limit of finitely generated free modules $L_\alpha,$ for $\alpha \in D$ some countable directed set.
By \cite[Application 3.5.10]{weibel-94}, see also the remark before that Application,  and the fact that the $L_\alpha$ are projective, we have
$$\Ext_{\OFG}^1(M; \prod\limits_{x\in {\OFG}_0}\Z[-,x]) \isom \varprojlim\nolimits^1 \Hom_{\OFG}(L_\alpha, \prod\limits_{x\in {\OFG}_0}\Z[-,x]).$$
In particular, $ \varprojlim\nolimits^1 \Hom_{\OFG}(L_\alpha, \prod\limits_{x\in {\OFG}_0}\Z[-,x])$ is a countable abelian group and hence 
Lemma \ref{lem:cat-gray-66} implies  that the inverse system $\Hom_{\OFG}(L_\alpha, R{\OFG})=L_\alpha^*$ satisfies the Mittag--Leffler condition in the weak sense. But since ${\OFG}$ is of finite isomorphism type, it satisfies the Mittag--Leffler condition in the strong sense.
 
Finally apply Proposition \ref{prop:raynaud-71}.
\end{proof}

\smallskip\noindent{\bf Proof of Theorem A.} This now follows directly from Proposition \ref{propA} by an easy dimension shift. \qed


\section{Groups of Bredon-homological dimension $1$}\label{Bredonhom-1}

In this section we will consider Bredon-cohomology for the family of finite subgroups only, and  we will  formulate a Bredon-analogue to Bieri's conjecture that every group of homological dimension $1$ is locally free, or equivalently, that every finitely generated group of homological dimension $1$ is free. The naive analogue to this conjecture is to ask whether every finitely generated group of Bredon-homological dimension $1$ is virtually free, but this is not the case as the lamplighter group shows~\cite{bridson-12}. Hence a rational version of Bieri's conjecture would need some further restrictions as well.
We propose a somewhat stronger assumption than finite generation, and suggest that the Bredon-analogue to finite generation is being of type Bredon-$\FP_1$. 

To stay in line with convention, we shall, in this section, denote the Bredon finiteness conditions by underlining. In particular, we put $\hd_{\Ffin} G = \uhd G,$ $\cd_{\Ffin} G = \ucd G,$ and denote by $\ul{\FP}_n$ the condition Bredon-$\FP_n$ for the family $\frakF=\Ffin.$

We ask the following question and devote the rest of the paragraph to giving some evidence for a positive answer, see Theorem \ref{hdf1}.

\begin{question}\label{conj:bredon-bieri}
Let $G$ be a group of type $\ul{\FP}_1$ and with $\uhd G=1$. Is~$G$  virtually free?

\end{question}

Note that for torsion-free groups this question is exactly the question posed by the finite generation version of Bieri's original conjecture. It immediately follows from \cite{dunwoody79} that an infinite virtually free group has $\hd_\frakF G =1$. Using Bass--Serre theory and a result by Karras, Pietrowski and Solitar, Dunwoody actually shows that a finitely generated group has $\cd_\Q G=1$ if and only if $G$ contains a free subgroup of finite index \cite[Corollary 1.2]{dunwoody79}. Furthermore, see \cite[Theorem 1.1]{dunwoody79}, a group~$G$ has $\cd_\Q G=1$ if and only if $G$ acts on a tree without inversions and with finite vertex stabilisers. This, in turn, implies that $\cd_\Q G =1$ if and only if $\ucd G =1.$ 

Also note that Question \ref{conj:bredon-bieri} has a positive answer for elementary amenable groups: they have Hirsch-length $1$, hence are locally-finite-by-(virtually torsion-free soluble) \cite[(g)]{wehrfritz}. Now the $\ul{\FP}_0$ condition implies that there is a bound on the orders of the finite subgroups. Hence the group is finitely generated and finite-by-virtually torsion-free soluble of Hirsch-length $1$ and thus  is virtually infinite cyclic.

\begin{remark}\label{lnremark}
In the above question one can  weaken the hypothesis of $\ul{\FP}_1$ to only asking for the group to be finitely generated and of type $\ul{\FP}_0$. There are examples of finitely generated groups (even $VF$-groups) that are of type $\ul{\FP}_0$ but not of type~$\ul{\FP}_1$, see \cite[Example 4]{ln-03}. These examples are, however, of Bredon-homological dimension at least $2$. Hence the question arises whether there are such examples of Bredon-homological dimension $1$.
\end{remark}

\begin{lemma}\label{coherent-lem}
Let $G$ be a virtually free group. Then the group ring $\Z G$ is coherent.
\end{lemma}

\begin{proof}
Since $G$ is virtually free, it has a finite index free subgroup $N$, whose group ring $\Z N$ is coherent. Now suppose we have a finitely generated $G$-module~$M$. Restriction to a finite index subgroup preserves finite generation, hence, as an $N$-module, $M$ is finitely presented. In particular, consider the short exact sequence of $G$-modules
$$ 0 \to K \to F \to M \to 0,$$ 
where $F$ is finitely generated free. As an $N$-module $K$ is finitely generated. Hence, the induced module $K \otimes_{\Z N} \Z G \cong K \otimes \Z[G/N]$ with the diagonal $G$-action is a finitely generated $G$-module. We also have a $G$-module epimorphism $\Z [G/N] \onto \Z$. Therefore $K \otimes \Z[G/N] \onto K,$ and $K$ is finitely generated as a $G$-module as required.
\end{proof}

Following \cite[Section 5]{emmanouil-12} we denote by $\frak S$ the class of groups for which the group ring $\Z G$ is a left $S$-ring; that is to say that every finitely generated flat $\Z G$-module is projective. We denote by $v\frak S$ the class of all groups containing a finite index subgroup belonging to $\frak S$.

Similarly, we define the class $\underline{\frak S}$ to be the class of groups for which every finitely generated flat ${\OFG}$-module is projective.

\begin{theorem}\label{hdf1}
Let $G$ be a group of type $\ul{\FP}_1$ and of $\uhd G=1$. Then the following are equivalent.

\begin{enumerate}
\item\label{hdf1-1} $G$ is virtually  free;
\item\label{hdf1-2} $G$ is virtually residually torsion-free nilpotent;
\item\label{hdf1-3} $G \in v\frak S;$
\item\label{hdf1-4} $G\in\underline{\frak S}$.
\end{enumerate}
\end{theorem}

\begin{proof}
$\ref{hdf1-1}\iff \ref{hdf1-2} \iff \ref{hdf1-3}$ follows directly from \cite[Proposition 5.5]{emmanouil-12}.  

$\ref{hdf1-1} \implies \ref{hdf1-4}$: By \cite[Lemma 3.1]{kropholler-09} it follows that for a group  of type Bredon-$\FP_n$,  the Weyl-group $WK=N_G(K)/K$ is of type $\FP_n$ for all $K \in \frakF$. Since $G$ is virtually free, it follows that $WK$ is virtually free as well. Hence by Lemma \ref{coherent-lem} the group rings $\Z(WK)$ are coherent. Now let $M$ be a finitely generated flat ${\OFG}$-module. Hence $M(G/K)$ is a finitely generated $\Z(WK)$-module. Since $\Z(WK)$ is coherent, it follows that $M(G/K)$ is finitely presented for all $K \in \frakF$. Since $G$ is, in particular, of type $\ul{\FP}_0$, \cite[Lemma 3.2]{kropholler-09} implies that $M$ is a finitely presented flat ${\OFG}$-module, and hence is projective \cite[Corollary 3.3]{nucinkis-04}.

$\ref{hdf1-4} \implies \ref{hdf1-1}$: Since $G$ is of type $\ul{\FP}_1$ we have  a short exact sequence:
$$0\to M(-) \to \prod_{K \in \calF}P_K(-) \to \Z(-) \to 0.$$
with  $\calF$  finite and $M(-)$  finitely generated. Since $\uhd G=1$, $M(-)$ is also flat. Now~$\ref{hdf1-4}$ implies that $M(-)$ is projective and hence $\ucd G =1$. This implies that $\cd_\Q G = 1.$
Using the fact that $G$ is finitely generated we can now apply \cite[Corollary 1.2]{dunwoody79} to get the claim.
\end{proof}


\section{Groups of rational homological dimension $1$}

In this final section we shall present a straightforward  generalisation of \cite[Proposition 3.2]{emmanouil-12}.

\begin{lemma}\label{countable-q-lem}
Let $G$ be a finitely generated group and $M$ be a countable $\Q G$-module. Then
$H^1_\Q(G, M)$ is a countable abelian group.
\end{lemma}

\begin{proof} The proof is completely analogous to that of \cite[Lemma 3.3]{emmanouil-12}.
\end{proof}

\begin{theorem}
Let $G$ be a finitely generated group of $\hd_\Q G=1$. Then the following are equivalent:
\begin{enumerate}
\item\label{thrm-p1} $G$ is virtually free.
\item\label{thrm-p2}$G$ has a non-trivial finitely generated free group as a normal subgroup.
\end{enumerate}
\end{theorem}

\begin{proof} \ref{thrm-p1} $ \implies$ \ref{thrm-p2}: It is a well known fact that if a group is virtually of property~$\mathcal P$, where $\mathcal P$ is a property that is closed under taking finite index subgroups, then there is a normal subgroup $N$ of finite index possessing this property $\mathcal P$.

\smallskip\noindent \ref{thrm-p2} $\implies$ \ref{thrm-p1}: It suffices to show that $H^2_\Q(G, \Q G)$ is a countable abelian group. An easy dimension shift applying Theorem \ref{main-et} together with Duwoody's result \cite{dunwoody79} then yields the claim.

Consider the Lyndon--Hochschild--Serre spectral sequence: 
$$H_\Q^p(G/N, H_\Q^q(N, \Q G)) \rightarrow H_\Q^{p+q}(G, \Q G).$$

Since $N$ is free, we have that $H^q(N, \Q G)=0$ for all $q\geq 2$. We also claim that $H^0_\Q(N, \Q G)=0$: Since $N$ is finitely generated free, it is of type $\FP_\infty$ over~$\Q$, and hence it suffices to show that $H^0_\Q(N, \Q N)=0$, which follows from \cite[Proposition 13.2.11.]{ross08}.

 This now implies that
$$H_\Q^2(G, \Q G)= H^1(G/N,H^1(N, \Q G)).$$
The result now follows from Lemma \ref{countable-q-lem}.
\end{proof}

\bibliographystyle{alpha}
\bibliography{bibdata}
\end{document}